\documentclass[10pt, reqno]{amsart}
\usepackage{amsmath, amsthm, amscd, amsfonts, amssymb, graphicx, color}
\usepackage[bookmarksnumbered, colorlinks, plainpages]{hyperref}
\hypersetup{colorlinks=true,linkcolor=red, anchorcolor=green, citecolor=cyan, urlcolor=red, filecolor=magenta, pdftoolbar=true}
\usepackage{mathrsfs}

\newtheorem{theorem}{Theorem}[section]
\newtheorem{lemma}[theorem]{Lemma}

\newtheorem{corollary}[theorem]{Corollary}
\theoremstyle{definition}
\newtheorem{definition}[theorem]{Definition}
\newtheorem{example}[theorem]{Example}

\newtheorem{conjecture}[theorem]{Conjecture}

\theoremstyle{remark}
\newtheorem{remark}[theorem]{Remark}
\numberwithin{equation}{section}

\begin{document}
\setcounter{page}{1}

\title[geometry of quaternions]{Birational geometry of quaternions}

\author[Nikolaev]
{Igor V. Nikolaev$^1$}

\address{$^{1}$ Department of Mathematics and Computer Science, St.~John's University, 8000 Utopia Parkway,  
New York,  NY 11439, United States.}
\email{\textcolor[rgb]{0.00,0.00,0.84}{igor.v.nikolaev@gmail.com}}

\dedicatory{All data are available as part of the manuscript}




\begin{abstract}
The Hilbert class field of the quaternion algebra  $B$
is  an  algebra $\mathscr{H}(B)$ such that every two-sided ideal of  $B$
is principal in  $\mathscr{H}(B)$.  
We study the avatars of $B$ and  $\mathscr{H}(B)$, i.e.  algebraic surfaces 
attached to  the quaternion algebras. 
It is proved that the avatar of $\mathscr{H}(B)$
is  obtained from the avatar of  $B$ by a birational map. 
We apply this result to the analogy between number fields and function fields. 
\end{abstract}

\maketitle

\section{Introduction}
Let  $B:=\left({a,b\over F}\right)$ be  the quaternion algebra over a field $F$,
see  [Voight 2021] \cite[Section 2.2]{V} or Section 2.2 for the notation.
 Following  \cite[Definition 1.1]{Nik0} (see also \cite{Nik1})  by  an  avatar  of  $B$ we understand an algebraic surface $S$,  such that
 \begin{equation}\label{eq1.1}
  \mathscr{A}_B\otimes\mathscr{K}\cong 
 \mathscr{A}_S,
 \end{equation}
where $\mathscr{A}_B$ is a $C^*$-algebra $\mathscr{A}_B$ defined by the norm 
closure of a self-adjoint representation 
of the ring $M_2(B)$ by the bounded linear operators on a Hilbert space $\mathcal{H}$, 
 $\mathscr{A}_S$ is the Serre $C^*$-algebra of surface $S$ \cite[Section 5.3.1]{N} and 
$\mathscr{K}$ is the $C^*$-algebra of compact operators on $\mathcal{H}$. 
For example,  the avatar of  the rational quaternions $\left({-1,-1\over \mathbf{Q}}\right)$
is the  complex projective plane $P^2(\mathbf{C})$. 
If $F\cong K$ is a number field,   then
the avatar of  $\left({a,b\over K}\right)$ is an algebraic surface $S$  covering the projective plane $P^2(\mathbf{C})$
with  ramification at  the three  knotted two-dimensional spheres  $P^1(\mathbf{C})\cup P^1(\mathbf{C})\cup P^1(\mathbf{C})$.
We refer the reader to Theorem \ref{thm2.7} for more details. 
It is known that the avatars of isomorphic quaternion algebras $B\cong B'$ correspond to the algebraic surfaces $S\cong S'$ isomorphic 
over the field $K$. In particular, any isomorphism $S\to S'$ is a continuous map.

 The birational geometry studies maps $S\dashrightarrow S'$ which  are no longer 
 continuous [Beauville 1996] \cite{B}. These are rational invertible  maps preserving 
 critical   data,  e.g. the field of rational  functions,  the plurigenera, the Kodaira dimension 
and the canonical ring of the surface $S$. It is well known that each birational map  $S\dashrightarrow S'$
is  composition of a finite number of the elementary birational maps called  the blow-ups.
In particular,  the blowing up of the surface $S$ along all  rational curves  with the self-intersection
index $-1$ yields a minimal model of $S$.  By the Castelnuovo Theorem, if $S$ is the minimal
model then any  birational map  $S\dashrightarrow S'$ is an isomorphism
[Beauville 1996] \cite[Theorem II.11]{B}.

The aim of our note  is a correspondence between the maps  $S\dashrightarrow S'$ and  the 
underlying quaternion algebras  $B$ and $B'$;  hence the name.
Our main result Theorem \ref{thm1.2}  says that $B'=\mathscr{H}(B)$ is the Hilbert
class field of $B$, see Definition \ref{dfn1.1}. 
Such a result has an application to the function field analogy 
[van der Geer,  Moonen \& Schoof 2005] \cite{GMS} and  [Rosen 2002] \cite{R};
this is discussed in Section 4.  To formalize our results we use  the following notation.

Let $K$ be a number field and $O_K$ its ring of integers. A principalization  is an extension 
$K\subseteq\mathscr{H}(K)$  such that every  ideal $I\subset O_K$  is principal in $\mathscr{H}(K)$, i.e. 
$IO_{\mathscr{H}(K)}\cong\alpha O_{\mathscr{H}(K)}$ for some element $\alpha\in \mathscr{H}(K)$. 
It is well known that  $\mathscr{H}(K)$ is  the Hilbert class field of $K$, i.e. the maximal abelian unramified 
extension of $K$;  hence the notation.

Let $B$ be a quaternion algebra and let $O_B$ be its ring of integers (maximal order)  [Voight 2021] \cite[Chapter 10]{V}. 
Denote by $I$ a two-sided ideal in the ring $O_B$  [Voight 2021] \cite[Chapter 18]{V}. 
\begin{definition}\label{dfn1.1}
The Hilbert class field of $B$ is a quaternion algebra $\mathscr{H}(B)$
containing $B$,  such that every two-sided ideal of the ring $O_B$  is a principal two-sided ideal of the 
ring  $O_{\mathscr{H}(B)}$. 
Equivalently,   $\mathscr{H}\left({a,b\over K}\right)\cong \left({a,b\over \mathscr{H}(K)}\right)$ 
(Lemma \ref{lm3.4}). 
\end{definition}
Our main result is as follows. 
\begin{theorem}\label{thm1.2}
The avatar of  $\mathscr{H}(B)$ is a blow-up of the avatar of quaternion algebra $B$, see diagram in  Figure 1. 
\end{theorem}
\begin{remark}
In a remarkable independent development, the quaternionic manifolds 
have been introduced in the series of papers 
 [Gentili,  Gori \& Sarfatti 2017] \cite{GenGorSar1}, 
 [Bisi \& Gentili 2018]    \cite{BisGen1} and 
 [Angella \& Bisi 2019]\cite{AngBis1};
 see also the monograph 
 [Gentili, Stoppato \& Struppa 2013] \cite{GSS}. 
 The key idea is the notion of a slice regular functions of one
 or several quaternionic variables providing the glue-up maps
 between the local charts of quaternionic manifolds  
 [Gentili, Stoppato \& Struppa 2013] \cite{GSS}.
 In particular, the blow-up maps were introduced and studied
   [Gentili,  Gori \& Sarfatti 2017] \cite[Section 3.5]{GenGorSar1}. 
 It was proved that such maps preserve the quaternionic structure
 of the manifolds  [Gentili,  Gori \& Sarfatti 2017] \cite[Theorem 3.12]{GenGorSar1}. 
 These results are fitting well our Theorem \ref{thm1.2}  saying that the blowing-up
 of the avatars induces a regular map between the quaternion algebras. 
 However, the notion of an avatar cannot be extended to the octonions;
 the latter are no longer associative unlike  the $C^*$-algebras. 
 \end{remark}
The paper is organized as follows.  A brief review of the preliminary facts is 
given in Section 2. Theorem \ref{thm1.2} is proved in Section 3. 
An application of Theorem \ref{thm1.2} to the function field analogy is discussed
in Section 4.

\begin{figure}
\begin{picture}(300,110)(-70,0)
\put(20,70){\vector(0,-1){35}}
\put(122,70){\vector(0,-1){35}}
\put(105,23){\vector(1,0){8}}
\multiput(45,23)(12,0){5}{\line(1,0){5}}
\put(45,83){\vector(1,0){60}}
\put(17,20){$S$}
\put(118,20){$S'$}
\put(17,80){$B$}
\put(105,80){ $\mathscr{H}(B)$}
\put(55,30){\sf blow-up}
\put(45,90){\sf  Hilbert class}
\put(60,70){\sf  field}
\end{picture}
\caption{}
\end{figure}

\section{Preliminaries}
This section is a brief review of the algebraic surfaces, quaternion algebras and their avatars.
We refer the reader to [Beauville 1996] \cite{B},  [Voight 2021] \cite{V} and \cite{Nik1} 
for a detailed account.

\subsection{Algebraic surfaces}
An algebraic surface is a variety $S$ of the complex dimension $2$. 
One can identify the non-singular algebraic surface $S$ with a  complex surface and therefore
with a smooth 4-dimensional manifold $\mathcal{M}$.  
The map $\phi: S\dashrightarrow S'$ is called rational,  if it is given  
by a rational function. 
The rational maps cannot be composed unless they are dominant, i.e. the image 
of $\phi$ is Zariski dense in $S'$. 
The map $\phi$ is birational,  if the inverse $\phi^{-1}$ is a rational map. 
A birational map  $\epsilon: S\dashrightarrow S'$ is called a blow-up,
if it is defined everywhere except for a point $p\in S$ and a rational curve $C\subset S'$,
such  that  $\epsilon^{-1}(C)=p$.   
Every birational map $\phi: S\dashrightarrow S'$ is composition of a finite 
number of the blow-ups, i.e. $\phi=\epsilon_1\circ\dots\circ\epsilon_k$.

The surface $S$ is called a minimal model, if any birational map $S\dashrightarrow S'$
is an isomorphism. Minimal models exist and are unique unless $S$ is 
a ruled surface. By the Castelnuovo Theorem, the smooth projective surface $S$ is a minimal 
model if and only if $S$ does not contain a rational curves $C$ with the 
self-intersection index $-1$.

\subsection{Quaternion algebras}
\begin{definition}
The algebra  $\left({a,b\over F}\right)$ over a field $F$ is called a quaternion algebra
  if there exists $i,j\in \left({a,b\over F}\right)$ such that $\{1,i.j, ij\}$ is a basis for $\left({a,b\over F}\right)$ and 
  \begin{equation}
  i^2=a, \quad j^2=b, \quad ji=-ij
  \end{equation}
   for some $a,b\in F^{\times}$. 
\end{definition}
\begin{example}
If $F\cong\mathbf{R}$ and $a=b=-1$, then the quaternion algebra 
$\left({-1,-1\over \mathbf{R}}\right)$ consists  of the Hamilton quaternions $\mathbb{H}(\mathbf{R})$;
hence the notation.  
If $F\cong\mathbf{Q}$, then the quaternion algebra 
$\left({-1,-1\over \mathbf{Q}}\right)$ consists  of the rational quaternions $\mathbb{H}(\mathbf{Q})$. 
\end{example}

\medskip
In what follows, we assume $F\cong\mathbf{Q}$ or an algebraic number field $K$. 
\begin{definition}
By the ring of integers $O_B$ of the quaternion algebra  $B=\left({a,b\over K}\right)$
we understand  the maximal lattice that is also a subring of $B$. 
\end{definition}

\medskip
A two-sided ideal $I_B\subseteq O_B$ is defined as an additive subgroup $I_B$ 
such that $rx\in I_B$ and $xr\in I_B$ for every $r\in O_B$ and $x\in I_B$. 
The following results establish a bijection between the two-sided ideals of $O_B$
and the ideals in the ring of integers $O_K$ of the ground field $K$.   
\begin{theorem}\label{thm2.4}
{\bf (\cite[Theorem 18.3.6]{V})}
The map: 
\begin{equation}\label{eq2.2}
\mathscr{P}_B\mapsto \mathscr{P}_B\cap O_K
\end{equation}
establishes a bijective correspondence between the prime two-sided ideals 
$\mathscr{P}_B$
of the ring $O_B$ and the prime ideals $\mathscr{P}_K\cong  \mathscr{P}_B\cap O_K$ of the ring $O_K$.  
\end{theorem}
\begin{corollary}\label{cor2.5}
The map (\ref{eq2.2}) extends to a bijection between all two-sided ideals of the ring $O_B$
and the ideals of the ring $O_K$.  
\end{corollary}
\begin{proof}
If $I_B\subseteq O_B$ is a two-sided ideal, then $I_B\cap O_K$ is an ideal of the ring $O_K$.
Since $O_K$ is the Dedekind domain, we conclude that $I_B\cap O_K$  is the product of the
prime ideals $\mathscr{P}_K$. Using (\ref{eq2.2}) one gets a bijection 
between all two-sided ideals of the ring $O_B$ and the ideals of the ring $O_K$. 
\end{proof}

\subsection{Avatars of quaternion algebras}
Let $O_B$ be ring of integers of the quaternion algebra
 $B=\left({a,b\over F}\right)$ and let $M_2(O_B)$ be the matrix ring over $O_B$.
Consider a self-adjoint representation
\begin{equation}
\rho: M_2(O_B)\to \mathscr{B}(\mathcal{H})
\end{equation}
 of the ring $M_2(O_B)$  by the bounded linear operators 
 on a Hilbert space $\mathcal{H}$.  Taking the norm-closure of $\rho(M_2(O_B))$ 
 in the strong operator topology,  one gets a $C^*$-algebra  $\mathscr{A}_B$.
 
 Let $S$ be a complex algebraic surface.   Denote by  $B(S, \mathcal{L}, \sigma)$  twisted 
homogeneous coordinate ring of $S$, where $\mathcal{L}$ is an invertible sheaf and $\sigma$
is an automorphism of $S$ [Stafford \& van ~den ~Bergh 2001]  \cite[p. 173]{StaVdb1}. 
Recall that the  Serre $C^*$-algebra, $\mathscr{A}_S$,   is the norm closure of a
self-adjoint representation of the ring  $B(S, \mathcal{L}, \sigma)$
in $\mathscr{B}(\mathcal{H})$  \cite[Section 5.3.1] {N}. 
Finally, let $\mathscr{K}$ be the $C^*$-algebra of all compact operators on the Hilbert space 
 $\mathcal{H}$. 
\begin{definition}\label{def2.6}
 The algebraic surface  $S$ is called an  avatar
 of the quaternion algebra $B$ if there exists  an isomorphism of the $C^*$-algebras: 
 \begin{equation}\label{eq2.4}
  \mathscr{A}_S\to\mathscr{A}_B\otimes\mathscr{K}.
 \end{equation}
\end{definition}

\medskip
The geometry of avatars  is described in the following result. 
\begin{theorem}\label{thm2.7}
{\bf (\cite{Nik1})}
Let $\left({a,b\over K}\right)$ be a quaternion algebra, where $K\subset\overline{\mathbf{Q}}$. 
 Then:

\medskip
(i) projective plane $P^2(\mathbf{C}) $ is an avatar of the rational quaternion algebra $\mathbb{H}(\mathbf{Q})$;

\smallskip
(ii)  the avatar of a quaternion algebra  $\left({a,b\over K}\right)$ is a non-singular
algebraic surface  $S(\overline{\mathbf{Q}})$;

\smallskip
(iii)  the field extension $\mathbf{Q}\subset K$ defines a covering
$S(\overline{\mathbf{Q}})\to P^2(\mathbf{C})$ ramified at 
  three  knotted two-dimensional spheres
$P^1(\mathbf{C})\cup P^1(\mathbf{C})\cup P^1(\mathbf{C})$.
\end{theorem}

\begin{remark}
Theorem \ref{thm2.7} (iii) can be viewed as an analog of Belyi's Theorem for the algebraic surfaces.
\end{remark}

\subsection{4-manifolds and cyclic division algebras}
Let $K$ be a number field and let $E$ be a finite Galois extension of $K$. 
Denote by $G=Gal ~(E|K)$ the Galois group of $E$ over $K$. 
Let $n=\dim_K (E)$ be the dimension of $E$ as a vector space over $K$. Consider 
the ring $End_K(E)$ of all $K$-linear transformations of $E$. Fixing a basis of $E$
over $K$, one gets an isomorphism
$End_K(E)\cong M_n(K)$. 
Denote by $\mathcal{C}\subset End_K(E)$ a subring generated by multiplications
by the elements $\alpha\in E$ and the automorphisms $\theta\in G$. 
It can be verified directly, that the commutation relation between the two is given by the 
formula $\theta\alpha=\theta(\alpha)\theta$. 
Further we restrict to the case when $G\cong \left(\mathbf{Z}/n\mathbf{Z}\right)^{\times}$ is a cyclic group of order $n$ 
generated by $\theta$.  Thus the relation $\theta\alpha=\theta(\alpha)\theta$   is complemented by the
 relation $\theta^n=1$.  On the other hand, it is easy to see that $\theta$ is an invertible element of $\mathcal{C}$ 
 along with any element of the form $\gamma\theta$, where $\gamma\in E$. Notice that 
$(\gamma\theta)^n=N(\gamma)\theta^n=N(\gamma)$, 
where $N(\gamma)\in K^{\times}$ is the $K$-norm of the algebraic number $\gamma$.   
\begin{definition}
The cyclic algebra $\mathcal{C}(a)$ is a subring of the ring $M_n(K)$ 
generated by  the elements $\alpha\in E$ and the element $u:=\gamma\theta$ satisfying the relations
$u\alpha=\theta(\alpha)u, \quad u^n=a\in K^{\times}$. 
\end{definition}
\begin{example}
Let $K\cong \mathbf{R}$ and $E\cong\mathbf{C}$. 
Then $G\cong\left(\mathbf{Z}/2\mathbf{Z}\right)^{\times}$ 
and  $\theta$ is the complex conjugation. In this case 
$\mathcal{C}(1)\cong M_2(\mathbf{R})$ and $\mathcal{C}(-1)\cong\mathbb{H}$,
where $\mathbb{H}$ is the algebra of real quaternions. 
\end{example}
The cyclic division algebras appear  in the framework of arithmetic topology \cite[Section 7.2]{N}.
Namely,  let  $\mathcal{M}$ be a  smooth 4-dimensional manifold and $\mathfrak{M}$ be the category
of such  manifolds so that the  arrows of $\mathfrak{M}$ are diffeomorphisms between the objects $\mathcal{M}$.
\begin{theorem}
{\bf  \cite[Theorem 7.4.5]{N}}
There exists a covariant functor $F:  \mathfrak{M}\to \mathfrak{C}$, where
$\mathfrak{C}$ is a category of the cyclic 
division algebras. 
\end{theorem}

\section{Proof}
For the sake of clarity, let us outline the main ideas.  Let $\left({a,b\over K}\right)$
be a quaternion algebra over the number field $K$ and let $S$ be its avatar. 
We identify $S$ with a smooth 4-dimensional manifold $\mathcal{M}$. 
Recall that there exists a functor $F$  \cite[Theorem 7.4.5 and Section 7.2]{N} mapping such manifolds to the cyclic
division algebras over a number field $K'$ \cite[Section 7.4.2]{N}.  Since 
 $\left({a,b\over K}\right)\subset M_4(K)$ \cite[Exercise 11, p. 32]{V}, we conclude that
 $K'\cong K$. 
 On the other hand,   the $K$-theory of  Etesi $C^*$-algebra $\mathbb{E}_{\mathcal{M}}$ of 
 $\mathcal{M}$ \cite[Section 7.5]{N}  is defined by  the number field $K$, see  
  \cite[Corollary 7.5.1(i)]{N}. 
   It remains to apply \cite[Theorem 8.6.1]{N} saying that 
 a blow-up of $\mathcal{M}$ corresponds to the Hilbert class field $\mathscr{H}(K)$ . 
 Let us pass to  a detailed argument.
We split the proof in a series of lemmas.

\medskip
\begin{lemma}\label{lm3.1}
If $S$ is an avatar of the quaternion algebra  $\left({a,b\over K}\right)$,  then
$F(S)$ is a cyclic  division algebra over $K$. 
\end{lemma}
\begin{proof}
(i)
Let  $S$ be an algebraic surface and $\mathcal{M}$ the underlying  smooth 4-dimensional manifold. 
Recall that in the framework of arithmetic topology \cite[Section 7.2]{N} one constructs a 
covariant functor $F:  \mathfrak{M}\to \mathfrak{C}$ , where $\mathfrak{M}$ is a category
of the smooth 4-dimensional manifolds and $\mathfrak{C}$ is a category of the cyclic 
division algebras  \cite[Theorem 7.4.5]{N}.

\medskip
(ii)  A cyclic division algebra $\mathcal{C}_K\in\mathfrak{C}$ over the  number field $K$
is defined as a division subring of the matrix ring $M_n(K),$ where $n=\dim_K(E)$ for 
a Galois extension $E$ of $K$ \cite[Section 7.4.2]{N}.   
It is easy to see that the non-split (\cite[Section 5.4]{V}) quaternion algebras  $\left({a,b\over K}\right)$
are  cyclic division algebras with  $E$ being a bi-quadratic  
extension of $K$.

\medskip
(iii)
Indeed, the non-split quaternion algebra  $\left({a,b\over K}\right)$
 can be represented as a subring of the matrix ring $M_4(K)$,
see  \cite[Exercise 11, p. 32]{V}.  Since $\left({a,b\over K}\right)$ is a
central simple algebra \cite[Corollary 7.1.2]{V}, it is a cyclic division algebra
over the field $K$ \cite[Section 7.5.5]{V}. 

\medskip
Lemma \ref{lm3.1} is proved.

\end{proof}

\medskip
\begin{remark}
The avatar map $A$ can be viewed as an  inverse of the map $F_0$ which is a restriction of the functor 
$F$ to  the category of algebraic surfaces $\mathfrak{S}$. 
This relation is shown in Figure 2, where $\mathfrak{Q}$ is the category of non-split quaternion algebras. 
\end{remark}
\begin{figure}
\begin{picture}(140,100)(0,0)
\put(40,70){\vector(0,-1){35}}
\put(30,35){\vector(0,1){35}}
\put(90,70){\vector(0,-1){35}}
\put(58,20){$\subset$}
\put(55,80){$\subset$}
\put(30,20){$\mathfrak{Q}$}
\put(88,20){$\mathfrak{C}$}
\put(30,80){$\mathfrak{S}$}
\put(80,80){ $\mathfrak{M}$}

\put(100,50){$F$}
\put(45,50){$F_0$}
\put(15,50){$A$}

\end{picture}
\caption{}
\end{figure}

\medskip
\begin{lemma}\label{lm3.3}
The cyclic division algebra $\mathcal{C}_{\mathscr{H}(K)}$ 
corresponds to a blow-up of the  surface
$S$, i.e. $F(S')=\mathcal{C}_{\mathscr{H}(K)}$, where $F(S)=\mathcal{C}_K$
and $\mathscr{H}(K)$ is the Hilbert class field of $K$. 
\end{lemma}
\begin{proof}
(i) Let $\mathcal{C}_K$ be a cyclic division algebra over the number field $K$.
Denote by  $\mathcal{M}_{top}$  a topological 4-dimensional manifold underlying 
the algebraic surface $S$.

\medskip
(ii)  Recall that the Brauer group $Br (K)$ of the number field  $K$ is an 
abelian group whose elements are Morita equivalence classes of the central
simple algebras over $K$. In particular, the group $Br (K)$ classifies the cyclic 
division algebras $\mathcal{C}_K$ over  $K$.

\medskip
(iii) On the other hand, the Brauer group $Br (K)$ is known to parametrize the smooth structures 
on  $\mathcal{M}_{top}$ \cite[Corollary 7.5.1(ii)]{N}. 
This fact is proved using the $K$-theory of operator algebras following the 
ideas of Gabor Etesi; we refer the reader to \cite[Section 7.5]{N} 
for the details and notation.  In particular, the Handelman 
triple $(\Lambda, [\mathfrak{m}], K)$ \cite[Section 3.5.2]{N} is a topological 
invariant of $\mathcal{M}_{top}$  \cite[Corollary 7.5.1(i)]{N}.

\medskip
(iv)  We can now apply \cite[Theorem 8.6.1]{N} saying that the birational map 
 $S\dashrightarrow S'$ is a blow-up if and only if $K'\cong \mathscr{H}(K)$. 
 Therefore the cyclic division algebra has the form $F(S')=\mathcal{C}_{\mathscr{H}(K)}$. 
 
 \bigskip
 Lemma \ref{lm3.3} is proved. 
\end{proof}

\medskip
\begin{lemma}\label{lm3.4}
$\mathscr{H}\left({a, b\over K}\right)\cong\left({a, b\over \mathscr{H}(K)}\right)$,
where $\mathscr{H}\left({a, b\over K}\right)$ is the Hilbert class field of the
quaternion algebra  $\left({a,b\over K}\right)$, see Definition \ref{dfn1.1}. 
\end{lemma}
\begin{proof}
(i) Let $K$ be a number field and $\mathscr{H}(K)$ its Hilbert class field,
i.e. maximal abelian unramified extension of $K$. Recall that every ideal in
the ring of integers $O_K$ of the field $K$ is a principal ideal in the ring of integers   $O_{\mathscr{H}(K)}$
of  $\mathscr{H}(K)$.
In other words, for every ideal $I_K\subset O_K$ there exists an  element $\alpha\in \mathscr{H}(K)$
such that:  
\begin{equation}\label{eq3.1}
 I_KO_{\mathscr{H}(K)} \cong \alpha O_{\mathscr{H}(K)}. 
\end{equation}

\medskip
(ii)  Let $B=\left({a,b\over K}\right)$ be a quaternion algebra and let $O_B$ 
 be its ring of integers.   Denote by $I_B\subseteq O_B$ a two-sided 
 ideal in the ring  $O_B$.  In view of Corollary \ref{cor2.5},  there exists a
 one-to-one map $h$  between the set $\{I_B\}$ of all two-sided ideal in $O_B$
 and the set $\{I_K\}$ of all ideals in the ring $O_K$ of the ground field $K$,
 i.e. we have a bijection
\begin{equation}\label{eq3.2}
 h: \{I_B\}\to \{I_K\}. 
\end{equation}

\medskip
(iii) Let us show that $\mathscr{H}(B):=\left({a, b\over \mathscr{H}(K)}\right)$
is  the Hilbert class field of the
quaternion algebra  $\left({a,b\over K}\right)$ given by Definition \ref{dfn1.1}. 
Indeed, repeating the argument of item (ii) one gets 
a bijection:
\begin{equation}\label{eq3.3}
 \{I_{\mathscr{H}(B)}\}\to \{I_{\mathscr{H}(K)}\}. 
\end{equation}
If $I_B\subseteq O_B$ is a two-sided ideal, then $h(I_B)=I_K\subseteq O_K$. 
From  (\ref{eq3.1}) one gets that the ideal $I_K$ is a principal in 
the ring $O_{\mathscr{H}(K)}$.   
Likewise,  using (\ref{eq3.3}) we conclude that $I_B$ is a principal two-sided ideal
in the ring $O_{\mathscr{H}(B)}$.  Thus  the quaternion 
algebra  $\left({a, b\over \mathscr{H}(K)}\right)$ is the Hilbert class field of the
quaternion algebra  $\left({a,b\over K}\right)$. 

\bigskip
Lemma \ref{lm3.4} is proved.  
\end{proof}

\medskip
\begin{remark}
Using  an extension formula for the base field
$\left({a, b\over\mathscr{H}(K)}\right)
\cong \left({a, b\over K}\right)\otimes_K \mathscr{H}(K)$
\cite[p. 22]{V},  one can write the result of  Lemma \ref{lm3.4} in a tensor form: 
\begin{equation}
\mathscr{H}\left({a, b\over K}\right)
\cong \left({a, b\over K}\right)\otimes_K \mathscr{H}(K). 
\end{equation}
\end{remark}

\medskip
\begin{corollary}\label{cor3.6}
 The diagram in Figure 1 is  commutative, i.e.  
 the avatar of quaternion algebra  $\mathscr{H}(B)$ is a
blow-up of the avatar of the quaternion algebra $B:=\left({a,b\over K}\right)$. 
\end{corollary}
\begin{proof}
Let $S$ be an avatar of  the quaternion algebra  $\left({a,b\over K}\right)$. 
To calculate the avatar of  $\mathscr{H}\left({a, b\over K}\right)$, recall from Lemma \ref{lm3.4} that 
$\mathscr{H}\left({a, b\over K}\right)\cong\left({a, b\over \mathscr{H}(K)}\right)$. 
 On the other hand, it follows from Lemmas \ref{lm3.1} and \ref{lm3.3} that 
the avatar of quaternion algebra $\left({a, b\over \mathscr{H}(K)}\right)$
is obtained by a blow-up  $S\dashrightarrow S'$ of the  surface $S$. 
Thus the avatar of the Hilbert class field  $\mathscr{H}\left({a, b\over K}\right)$ 
 is the algebraic surface $S'$.  Corollary \ref{cor3.6} is proved.
\end{proof}

\bigskip
Theorem \ref{thm1.2} follows from  Corollary \ref{cor3.6}.

\section{Function field analogy}
It was first observed in [Dedekind \& Weber 1882] \cite{DedWeb1}
that  number fields and  fields of rational functions in one variable
are related. For example,   both fields  are   Dedekind domains, i.e.  every ideal factors 
into a product of the prime ideals.  Unlike number fields,   the group of units of the function fields
 is no longer finitely generated.  However, the behavior is back to normal for 
 the rational functions over finite (Galois)  fields. 
In fact,  it was proved in [Artin \& Whaples 1945]  \cite{ArtWap1} that 
the two fields fit into  the concept of a ``global field''. 
This led to a fast mutual progress achieved by recasting the  open 
problems on either side and solving them in terms of the other [van der Geer,  Moonen \& Schoof 2005] \cite{GMS}.
The successful proof  of an  analog of the Riemann Hypothesis (Hasse \& Weil) and the geometric Langlands Conjectures (Drinfeld \& Lafforgue)
are a few   examples.  When and if  such an analogy can be upgraded to an isomorphism between (certain) number fields and function fields 
is an interesting open problem.
Below we suggest an approach  based on Theorem \ref{thm1.2}.

\begin{figure}[h]
\begin{picture}(150,150)(0,0)

\put(70,100){$C$}

\put(55,90){\vector(-1,-1){20}}
\put(95,90){\vector(1,-1){20}}
\put(70,53){\vector(1,0){20}}

\put(75,55){$\sim$}

\put(10,50){$\mathscr{H}^{m-1}(K)$}
\put(110,50){$\mathbf{F}_q(C)$}

\end{picture}
\caption{
}
\end{figure}

Let  $P=(x_0,y_0)$ be an isolated singularity of the algebraic curve 
\begin{equation}\label{eq4.1}
 f(x,y)=0,
\end{equation}
where $f\in \mathbf{Z}[x,y]$.  Consider an arithmetic scheme: 
\begin{equation}\label{eq4.2}
 Spec ~\mathbf{Z}[x,y]/(f)
\end{equation}
and its morphism to $Spec~\mathbf{Z}$.  Let $(P)$ be a prime ideal 
of (\ref{eq4.2}) lying over $(p)\in Spec~\mathbf{Z}$. 
Recall that the resolution of singularity $(x_0,y_0)$ is a birational map  
\begin{equation}\label{eq4.3}
\phi=\epsilon_1\circ\dots\circ\epsilon_{m-1}
\end{equation}
composed of the blow-ups $\epsilon_i$  bringing   (\ref{eq4.1})
to a non-singular algebraic curve $C$.  Let $q$ be a power of $p$ and 
we denote by $\mathbf{F}_q(C)$  the geometric extension of the field  of rational functions in one variable 
 defined by the curve $C$ [Rosen 2002] \cite[Chapter 7]{R}.

On the other hand,  an analog of Theorem \ref{thm1.2} for  the algebraic curves  suggests that  each  blow-up $\epsilon_i$
is the Hilbert class field $\mathscr{H}(K)$ of the underlying number field $K$.  Thus one gets a class field 
tower:
\begin{equation}\label{eq4.4}
K\subset\mathscr{H}(K)\subset\dots\subset \mathscr{H}^{m-1}(K),
\end{equation}
where   the class number $|Cl~(\mathscr{H}^{m-1}(K))|=1$ and $\mathscr{H}^{i+1}(K):=\mathscr{H}(\mathscr{H}^i(K))$. 
Consider  a pair of the global fields $\mathscr{H}^{m-1}(K)$ and $\mathbf{F}_q(C)$   attached to 
the curve $C$,  see Figure 3.  The following conjecture says that the correspondence is a functor.

\begin{conjecture}\label{cnj4.1}
$\mathscr{H}^{m-1}(K)\cong \mathbf{F}_{q}(C),$ where  $q=p^{m-1}$.
\end{conjecture}


\section*{Data availability}
  
  Data sharing not applicable to this article as no datasets were generated or analyzed during the current study.
   
\section*{Conflict of interest}
On behalf of all co-authors, the corresponding author states that there is no conflict of interest.
  

\subsection*{Acknowledgment}
 The author would like to thank the anonymous referee who provided thoughtful  comments on an earlier version of the manuscript.

\bibliographystyle{amsplain}


\end{document}